\documentclass[12pt]{amsart}
\usepackage{amssymb,latexsym}
\usepackage{enumerate}
\usepackage[margin=2in]{geometry}
\usepackage{graphicx}
\usepackage{psfrag}
\makeatletter
\@namedef{subjclassname@2010}{%
  \textup{2010} Mathematics Subject Classification}
\makeatother

\newtheorem{thm}{Theorem}[section]
\newtheorem{corollary}{Corollary}
\newtheorem{proposition}{Proposition}
\newtheorem{lemma}{Lemma}

\theoremstyle{definition}

\newtheorem{exa}{Example}

\newtheorem*{xrem}{Remark}

\numberwithin{equation}{section}

\newcommand\R{\mathbb{R}}

\newcommand\n{\mathcal{N}}
\newcommand\m{\mathcal{M}}
\newcommand\id{\mathrm{id}}
\newcommand{\sr}[1]{\mathfrak{M}_{#1}}

\newcommand{\norma}[1]{\left\| #1 \right\| }

\begin{document}


\baselineskip=17pt



\title[Scales of Q-A means determined by invariance property]{Scales of Quasi-Arithmetic means determined by invariance property}

\author{Pawe{\l} Pasteczka}
\address{Institute of Mathematics\\University of Warsaw\\
02-097 Warszawa, Banach str. 2, Poland}
\email{ppasteczka@mimuw.edu.pl}

\date {May 31, 2014}

\begin{abstract}
It is well known that if $\mathcal{P}_t$ denotes a set of power means then the mapping 
$\mathbb{R} \ni t \mapsto \mathcal{P}_t(v) \in (\min v, \max v)$
is both 1-1 and onto for any non-constant sequence $v = (v_1,\dots,\,v_n)$ of positive numbers.
Shortly: the family of power means is a scale.

If $I$ is an interval and $f \colon I \rightarrow \mathbb{R}$ is 
a continuous, strictly monotone function then $f^{-1}(\tfrac{1}{n} \sum f(v_i))$ is 
a natural generalization of power means, so called quasi-arithmetic mean generated by $f$.

A famous folk theorem says that the only homogeneous, quasi-a\-rith\-me\-tic means are power means.
We prove that, upon replacing the homogeneity requirement by an invariant-type axiom, one gets a family
of quasi-arithmetic means building up a scale, too.
\end{abstract}



\maketitle

\section{Introduction}
In the theory of means, since very long, the most popular family 
of means consists of the power means. In the beginning of the 1930s 
several authors \cite{deFinetti,kolmogoroff,nagumo} proposed 
a natural generalization of this family -- so-called 
quasi-arithmetic means.

The quasi-arithmetic mean is defined for any continuous, strictly 
monotone function $f \colon U \to \R$, where $U$ is an open interval. 
When $v = (v_1,\dots,\,v_n)$ is a sequence of points in $U$ 
and $w = (w_1,\dots,\,w_n)$ is a sequence of weights 
($w_i > 0$, $w_1 + \cdots + w_n = 1$), then the mean 
$\sr{} = \sr{f}(v,\,w)$ is 
well-defined by the equality 
$$
f(\sr{}) = \sum_{i=1}^n w_i f(v_i)\,.
$$ 

Natural considerations lead to, at first, comparison
of means, what means that $\sr{f}(v,\,w)\geq
\sr{g}(v,\,w)$ for arbitrary $v$ and 
$w$ with equality only if $v$ is a constant vector
(such a means is called {\it comparable}). 
And, as natural implication, looking for families 
of means, such that each two of them are comparable.
In fact, a family of functions $\{f_i \colon 
U \to \R\}_{i \in I}$, where $I$ and $U$ are open intervals, is called a 
{\it scale on $U$} if for every non-constant vector $ v \in U^n$ 
and arbitrary fixed corresponding weights $ w$, the mapping 
$I \ni i \mapsto \sr{f_i}(v,\,w)$ is a {\it bijection\,} 
onto the interval $(\min v,\,\max v)$.
Let us note that a scale is a maximal family of comparable means,
but not all maximal families of comparable means are scales.

In the previous paper~\cite{moja} we considered the problem when a given 
family of quasi-arithmetic means was actually a scale. The main result 
furnished some necessary, and, alongside, also some sufficient conditions 
for family's being a scale. It follows that there exist much more scales 
of quasi-arithmetic means than previously thought and/or encountered in 
the literature. The scales appear to be fairly ubiquitous. 

One of the most natural questions is about some possible other axiom(s) to be 
imposed, with the aim to boil scales down to only those simple in some prescribed 
sense. A hint in this direction comes from an old folk theorem in the theory 
of means. It says that the quasi-arithmetic means that are {\it homogeneous\,} 
are just the power means. In the present paper we would like to present analogous 
results for axioms similar to homogeneity. An idea that prompts by itself is 
to use neutral functions and the notion of invariance. 

\section{Neutral functions and invariance}
Let $f \colon U \to \R$ be a continuous, strictly monotone function, 
$V \subset U$ be a subinterval. We say that a function $\n \colon V \rightarrow U$ 
is neutral for the mean $\sr{f}$ (or simply \emph{$f$-neutral}) if it is continuous, 
1-1 and  
\begin{equation}\label{eq:neutral_def}
\sr{f}(\n(v),\,w)=\n(\sr{f}(v,\,w)). 
\end{equation}
for an arbitrary vector $v \in V^n$ and corresponding 
arbitrary weights $w$. In such a situation the mean $\sr{f}$ 
is called \emph{invariant under $\n$}.

It is interesting to note that~\eqref{eq:neutral_def} alone implies 
that $\n$ is continuous and monotone. Indeed, let us fix $v_1,v_2 \in V$, \,$v_1 < v_2$. 
Then the continuity of $\n$ in the interval $(v_1,v_2)$ becomes obvious through 
the identity
$$
\sr{f}\left((\n(v_1),\n(v_2)),(w,1 - w)\right) = \n\left(\sr{f}\left((v_1,v_2),(w,1 - w)\right)\right).
$$

Now let $\n(v_1) < \n(v_2)$ and $v_0 \in (v_1,v_2)$. Then 
$\sr{f}\left((v_1,v_2),(w,1 - w)\right) = v_0$ for some $w \in (0,1)$.

\begin{align}
\textrm{Hence }\n(v_0)&=\n\left(\sr{f}\left((v_1,v_2),(w,1-w)\right)\right)\nonumber \\
&=\sr{f}\left((\n(v_1),\n(v_2)),(w,1-w)\right).\nonumber
\end{align}
So $\n(v_0) \in \left(\n(v_1),\,\n(v_2)\right)$ and $\n$ is increasing. Similarly, 
if $\n(v_1) > \n(v_2)$ then $\n$ is decreasing. Hence $\n$ is strictly monotone.

Moreover, it follows from the above calculations that 
\begin{align}
\sr{f}\left((v_1,v_2),(w,1 - w)\right) &= \n^{-1}\left(\sr{f}\left((\n(v_1),\n(v_2)),(w,1 - w)\right)\right)\nonumber\\
&= \sr{f \circ \n} \left((v_1,v_2),(w,1 - w)\right). \nonumber
\end{align}
This equation holds for arbitrary arguments, so 
$\sr{f} = \sr{f \circ \n}$ for the set of two variables and weights. By \cite[pp.\,66--68]{HLP}, it implies that there exists $a\neq 0$ and $b$ satisfying $f\circ \n(x) = af(x) + b$ for all $x$ in $V$. 
Hence 
\begin{equation}\label{eq:postacn}
\n(x) = f^{-1}\left(af(x) + b\right). 
\end{equation}

By the same source we know that for such a function $\n$ the equality \eqref{eq:neutral_def} holds.

In the consideration of quasi-arithmetic means determined by invariance properties it is 
natural to describe a class of of means, which are invariant under particular function $\n$.
As we will recover later, this class consists of many means. Some of them might be not 
comparable between each other. Such a class is obviously not a scale.
This blame might be easily observe in the following
\begin{exa}
\label{ex:perod2pi}
Let $f(x) = x$, $g(x) = x + \omega(x)$, where $\omega$ is a $1$-periodic, 
$C^1$ function such that $\norma{\omega'}_{\infty} < 1$. The function $x \mapsto x + 1$ 
is both $f$-neutral and $g$-neutral, but the means $\sr{f}$ and $\sr{g}$ may not be comparable.
\end{exa}

To eliminate this drawback we go back to the famous result that the only homogeneous,
quasi-arithmetic means are power means (this result will be proved independent in 
the Proposition~\ref{prop:powermeans}). As therein, we suppose that a family of 
functions is given and we consider a class of means which are invariant under all of them.

The other type of problem turns on if we put $a=-1$ in~\eqref{eq:postacn} or, equivalently,
assume that $\n \circ \n = \id$. For example it is easy to observe that 
$\sr{f}$ is invariant under $\n(x)=-x$ if $f$ is an arbitrary odd function.
This class is, however, too big to have interesting properties. To deal with
it from now on we will suppose that $a \ne -1$ or, equivalently, $\n \circ \n \ne \id$.

As it was already announced, for $f$-neutral function $\n$ we construct a set, containing $\n$, of functions
invariant under $f$. To build this set it is quite natural to say $\m$ to be \emph{$f$-root of $\n$} 
if
\begin{itemize}
\item $\m^i=\n$ for some $i \in \mathbb{Z}$,
\item $\m$ is $f$-neutral,
\item $\m$ and $\n$ are of the same monotonicity.
\end{itemize}

\noindent 
At this moment a natural assumption emerges that \textsf{a mean is invariant under $\n$ 
and all its $f$-roots}. 
Sometimes such an invariance coincide with some of classical axioms put forward in \cite[p.62]{bullen}. 
This situation might be observe in the following
\begin{exa}
\label{ex:hom}
A mean $\sr{}$ is invariant under the mapping $x \mapsto 2x$ and all its $\id$-roots if and only if it is homogeneous. 
\end{exa}
\begin{proof}
Of course if $\sr{}$ is homogeneous then it is invariant under the mapping
$x\mapsto 2x$ and all its $\id$-roots. The proof of the second 
implication uses an elementary group theory. 

Let us denote by $G$ the set of all functions neutral for $\sr{}$.
By the definition of neutral functions $\mathbb{G}=(G,\circ)$ is a group.
To demonstrate homogeneity of $\sr{}$ we, equivalently,
prove that $\mathbb{H}=(\{x \mapsto ax,a \in \mathbb{R}_{+}\},\circ)$ is a subgroup of $\mathbb{G}$.

Both $\mathbb{G}$ and $\mathbb{H}$ are closed subset of the family of all strictly increasing functions in the topology of the almost uniform convergence. 
It is easy to verify, that all mappings $x \mapsto 2^{\frac{1}{n}}x$ belong to $\mathbb{G}$ for an arbitrary $n \in \mathbb{N}$ as $n$-th $\id$-roots of the mapping $x \mapsto 2x$. 
Hence the group $\mathbb{K}$ generated by this mappings is a subgroup of $\mathbb{G}$. Since 
$$\mathbb{K}=(\{2^{q}x \colon q \in \mathbb{Q}\},\circ)$$
is a dense subset of $\mathbb{H}$ in considered topology and
$\mathbb{G}$ is a closed set containing $\mathbb{K}$, we know that $\mathbb{G}$ contains $\mathbb{H}$ as well.
\end{proof}
Therefore our main theorem presents as follows

\begin{thm}
\label{thm:invscale}
Let $U$ and $U_0$ be intervals, $U_0 \subset U$, $f \colon U \rightarrow \mathbb{R}$ be a continuous,
strictly monotone function, $\n \colon U_0 \rightarrow U$ -- $f$-neutral function, $\n \circ \n \neq \id$.
Then the family of all means invariant under $\n$ and all its $f$-roots is a scale on $U$.
\end{thm}

A proof of this theorem is given in Section~\ref{sec:invscale}.
First, we make some preparation staff.

\section{Preliminaries to the proof of Theorem~\ref{thm:invscale}}
The meaning of this section is to figure out and deal with
technical details of the proper main theorem's proof, 
given in next section.

\subsection{Roots}
First, we deal with existing of roots. 
Let us observe that a situation, where the function has no square root 
is possible (for example if it is decreasing). 
Fortunately there holds the following

\begin{lemma}[cube root]
\label{lem:3sq}
Let $f$ and $\n$ like in Theorem~\ref{thm:invscale}. Then there exists 
an unique $f$-neutral function $\m$ satisfying $\m \circ \m \circ \m = \n$.
\end{lemma}
\begin{proof}
Let $\n(x)=f^{-1}(af(x)+b)$. We need to find $p,q$ such that $\m(x)=f^{-1}(pf(x)+q)$.
But $\m^3(x)=f^{-1}(p^3f(x)+(p^2+p+1)q)$.
Hence 
$$a=p^3 \qquad \textrm{, and } \qquad b=(p^2+p+1)q.$$
So, lastly
\begin{equation}
p=a^{1/3} \qquad \textrm{, and } \qquad q=\frac{b}{a^{2/3}+a^{1/3}+1}. \label{eq:3sq}
\end{equation}
\end{proof}
\begin{xrem} 
$\m$ depends on both $\n$ and $f$. 
\end{xrem}

In fact we might observe that in the set of roots there exists a functions 
approximating identity. More precisely, we get the following
\begin{corollary}
\label{col:3sqrt}
Let $f$ and $\n$ like in Theorem~\ref{thm:invscale}.
Then there exists a sequence $\n_1,\n_2,\ldots$ of $f$-neutral functions, $\n_i \rightarrow \id$ and
$\n_i^{3^i}=\n$ for all $i\in \mathbb{N}$.
\end{corollary}
The proof is just using Lemma~\ref{lem:3sq} and the fact that the limit
$\lim_{a \rightarrow 1} \tfrac{\sqrt[3]{a}-1}{a-1}$ is a positive number less than $1$.

\subsection{Domain}
Let $f\colon U \rightarrow \R$ and $\n \colon U_0 \rightarrow U$ be an $f$-neural function.
Let us observe that the set $U_0$ and $\n(U_0)$ might be an arbitrary small, open subinterval of $U$.
But, as $U_0$ is an infinite set, the function $\n$ satisfies equation \eqref{eq:postacn} 
for exactly one pair of constant ($a$ and $b$).
Therefore a function $\n$ might be expand to maximal possible interval in the exactly one way. 

Hence we will assume that the domain $U_0$ consists of all $x$, for which 
the equation \eqref{eq:postacn} is well defined. 
This kind of expansion appears in the background of the presented proof,
but as long as it will be obvious, it will not bother ourselves.

\subsection{Functional equations}

As we already announce, at the moment we deal with the following, technically crucial, lemma.
This type of equations will appear in the proof of our main theorem.
\begin{lemma}
\label{lem:srnie}
Let $(p_i),(a_i)$ be two sequences of positive numbers, $(a_i) \rightarrow 1$, $a_i \neq 1$.
If $m \colon \R_{+} \rightarrow \R$ is a continuous function satisfying 
\begin{equation}
m(a_iy)=p_im(y)\textrm{ for all }y \in \mathbb{R}_+\textrm{ and }i \in \mathbb{N} \label{eq:srnie}
\end{equation}
then $m(x)=\alpha x^\beta$ for some $\alpha$ and $\beta$
\end{lemma}
\begin{proof}
Let us take a positive number $x$. Then 
$$x=\prod_{i=1}^{\infty}a_i^{c_i}$$
for some series $(c_i)$ of integer numbers. 

If $m(1)=0$ then $m(a_i)=0$ and $m(1/a_i)=0$ for all $i \in \mathbb{N}$, so $m \equiv 0$.
Otherwise we may multiply $m(x)$ by an arbitrary constant keeping~\eqref{eq:srnie} holds. 
Hence we may suppose, with no loss of generality, that $m(1)=1$.

Now, taking $y=1$ implies $m(a_i)=p_i$, so $m(a_iy)=m(a_i)m(y)$. Moreover $1=m(1)=m(a_i)m(1/a_i)=p_im(1/a_i)$ so $m(1/a_i)=1/p_i$.
Next, by an easy induction, we obtain 
$$m\left(y\prod_{i=1}^{N}a_i^{c_i}\right)=m(y)\prod_{i=1}^{N}m(a_i)^{c_i}=m(y)m\left(\prod_{i=1}^{N}a_i^{c_i}\right)$$
for an arbitrary $N$. As $N \rightarrow \infty$ we immediately get 
$$m(yx)=m(y)m(x).$$
Hence by \cite[pp. 37--40]{aczel} $m(x)=x^\beta$ for some $\beta \in \R$. Multiplying by a constant number ends the proof.
\end{proof}

\begin{xrem}
Similarly if $(p_i)$ and $(a_i)$ are two sequences, $(a_i) \rightarrow 0$, $a_i \neq 0$ and 
$m(a_i+y)=p_i+m(y)$ for all $y \in \mathbb{R}$ and $i \in \mathbb{N}$ then 
$m(x)=\alpha+\beta x$ for some $\alpha$ and $\beta$.
\end{xrem}

\subsection{Scales}
We also separate from the proof of the main theorem the fact that family of function generates a scale. 
In fact only one type of scale appear there. More precisely we need to prove the following
\begin{lemma}
\label{lem:fam_likePM}
Let $U$ be an interval, $g \colon U \rightarrow \R_{+}$ be a continuous, strictly monotone function. Then the family
$\{g(x)^\beta \colon \beta \neq 0 \} \cup \{ \ln g(x) \}$ generates a scale on $U$.
\end{lemma}
\begin{proof}
Let us consider a family of functions $m_\beta \colon \R_+ \rightarrow \R$, 
$$m_\beta(x)=\begin{cases} x^\beta & \beta \neq 0 \\ \ln x & \beta=0 \end{cases}.$$
We know that the family $m_\beta$ generates a scale on $\R_{+}$.\footnote{
Fact that the family of power means is a scale is one of the most classical 
result in the theory of means. It was already proved in the 19th century.}
Let us fix a non-constant vector $a \in U^n$ and corresponding weights $w$. 
Then the mapping 
$$\R \ni \beta \mapsto \sr{m_\beta \circ g}(a,w)=g^{-1} \Bigl(\sr{m_\beta}\bigl(\vec{g}(a),w\bigr)\Bigr)$$
is 1-1 and its image equals
$$\vec{g^{-1}}\Bigl(\min\bigl(\vec{g}(a)\bigr),\max\bigl(\vec{g}(a)\bigr)\Bigr)=
\Bigl(\min a,\max a \Bigr).$$
Hence the family $\{m_\beta \circ g \colon \beta \in \R\}$ generates a scale on $U$.
\end{proof}

\section{\label{sec:invscale} Proof of Theorem~\ref{thm:invscale}}
In view of Corollary~\ref{col:3sqrt} we may take the family $\{\n_i \colon U_i \rightarrow U\}$ 
-- $f$-roots of $\n$, $\bigcup U_i =U$, $\n_1=\n$. 
Then $\sr{f}$ is invariant under $\n_i$ for an arbitrary $i$.

From~\eqref{eq:postacn} $\n(x)=f^{-1}(af(x)+b)$ and, similarly, 
for each $i$ there exists $a_i$ and $b_i$ such that 
$$\n_i(x)=f^{-1}(a_if(x)+b_i)$$ 
(we use convention that $\n_0=\n$, $a_0=a$ and $b_0=b$).
Moreover, to make calculations more clear, we denote by $\simeq$ the fact that there exists nontrivial affine transformation of each side such that the equality holds.

As $\n_i \rightarrow \id$ we get $(a_i,b_i) \rightarrow (1,0)$. 
Let $g$ be such a function that $\sr{g}$ is invariant under $\n_i$ for all $i$. 
Then we obtain a family of equalities
$$f^{-1}(a_if(x)+b_i)=g^{-1}(p_ig(x)+q_i) \textrm{, for some }p_i\textrm{ and }q_i.$$
Meaning of the pairs $\{(a_i,b_i) \colon i \in \mathbb{N}\}$ and $\{(p_i,q_i) \colon i \in \mathbb{N}\}$
is very close. Hence it is useful to adopt similar conventions; $p=p_0$ and $q=q_0$.

As we will recover later, behaviour of the equation above radically depends on the fact if 
$a$ equals to $1$ (similarly $p$ equals to $1$). This fact determines natural partition of the proof into
two main cases (depends on value of $a$) and two subcases (depends on value of $p$)
within each of them.
In each of subcase turns on the same idea -- modify the equality equivalently to make
application of Lemma~\ref{lem:srnie} meaningful and then turn back to get the final function $g(x)$.

The prove is not too hard, but we pinpoint all calculations
because, as we will see, some of parameters will vanish as affine transformation
and some of them blow up into a family of means.

\subsection{\label{subsec:case1}Case $a \neq 1$.}
We may consider $f(x)=f_1(x)-\frac{b}{a-1}$ using fact that $\n(x)=f_1^{-1}(af_1(x))$ and $\sr{f}=\sr{f_1}$ we may suppose, without loss of generality, that $b \equiv 0$. 

Due to~\eqref{eq:3sq} we get $b_i \equiv 0$ for all $i$ as well. 
To ensure that all possible cases is considered we add arbitrary constant to $f$ in the end of this case.
As $b_i \equiv 0$ and $\n_i \neq \id$ one get that $a_i \neq 1$ for all $i$.

Since for all $i$ a function $\n_i$ is neutral for $\sr{g}$, there exists $p_i$ and $q_i$ such that for all $x \in J$ 
\begin{align}
f^{-1}(a_if(x))&=g^{-1}(p_i g(x)+q_i), \nonumber \\
g\circ f^{-1}(a_iy)&=p_ig\circ f^{-1}(y)+q_i, \nonumber \\
m(a_iy)&=p_im(y)+q_i \textrm{, where }m(y)=g\circ f^{-1}(y). \nonumber
\end{align}

Let us consider two subcases:
\subsubsection{Subcase $p\neq 1$} 
As in the proof of Lemma~\ref{lem:3sq} we have 
$$\frac{q_{i+1}}{p_{i+1}-1}=\frac{q_i}{p_i^{2/3}+p_i^{1/3}+1}\frac{1}{p_i^{1/3}-1}=\frac{q_i}{p_i-1}.$$
So $\frac{q_i}{p_i-1}=\frac{q}{p-1}$. Hence, we get step by step

\begin{align} 
m(a_iy)+\frac{q_i}{p_i-1}&=p_i\left(m(y)+\frac{q_i}{p_i-1}\right), \nonumber \\
m(a_iy)+\frac{q}{p-1}&=p_i\left(m(y)+\frac{q}{p-1}\right), \nonumber \\
m(y)+\frac{q}{p-1}&=\alpha y^\beta \textrm{, (by Lemma~\ref{lem:srnie})} \nonumber \\
m(y)&=\alpha y^\beta-\frac{q}{p-1}, \nonumber \\
g(x)&= \alpha f(x)^\beta-\frac{q}{p-1}, \nonumber \\
g(x)&\simeq f(x)^\beta. \nonumber
\end{align}
\subsubsection{Subcase $p=1$}
We obtain $p_i=1$ for an arbitrary $i$. So
\begin{align}
m(a_iy)&=q_i+m(y), \nonumber \\
\exp\left(m(a_iy)\right) &=\exp\left(m(y)\right)\exp(q_i),\nonumber \\\
\exp\left(m(y)\right)&=\alpha y^{\beta} \textrm{, (by Lemma~\ref{lem:srnie}),} \nonumber \\
m(y)&=\ln(\alpha)+\beta \ln y, \nonumber\\
g(x)&\simeq \ln f(x). \nonumber 
\end{align}
Concluding Case~\ref{subsec:case1} if $\n(x)=f^{-1}(af(x))$ and all its $f$-roots is a neutral function 
then the generator of mean is one of the family
$\{ f(x)^{\beta} \colon \beta \neq 0 \} \cup \{ \ln f(x) \}$.

\begin{xrem}
The second implication, for an arbitrary $a$ the function $f^{-1}(af(x))$
is neutral for each mean generated by the function belong to the family above, 
becomes very simple in view of~\eqref{eq:postacn}. 
\end{xrem}

Moreover one may add to $f$ an arbitrary real number (see the beginning of this case). 
Then, taking absolute value, we have families of generators of two types
\begin{align}
\{ (f(x)+q)^{\beta} \colon \beta \neq 0 \} &\cup \{ \ln (f(x)+q) \} \nonumber \\
&\, \textrm{, where } q \in \R, x \in \{x \in U \colon f(x)+q>0\}. \label{F11} \\
\{ (-f(x)+q)^{\beta} \colon \beta \neq 0 \} &\cup \{ \ln (-f(x)+q) \} \nonumber \\
&\, \textrm{, where }  q \in \R, x \in \{x \in U \colon -f(x)+q>0\}. \label{F12} 
\end{align}

\subsection{\label{subsec:case2}Case $a = 1$}
Then $\n(x)=f^{-1}(f(x)+b)$. Similar like in the above calculations let us take $\n_i(x)=f^{-1}(f(x)+b_i)$.
But $\n_i \neq \id$, so $b_i$ is a nonzero number. We get the following equalities
\begin{align}
f^{-1}(f(x)+b_i)&=g^{-1}(p_ig(x)+q_i), \nonumber \\
g\circ f^{-1}(y+b_i)&=p_ig\circ f^{-1}(y)+q_i, \nonumber \\
m(y+b_i)&=p_im(y)+q_i \textrm{, where }m(y)=g\circ f^{-1}(y). \nonumber
\end{align}

\noindent Let us consider two subcases:
\subsubsection{Subcase $p\neq 1$}
By the proof of Lemma~\ref{lem:3sq} we have 
$$\frac{q_{i+1}}{p_{i+1}-1}=\frac{q_i}{p_i^{2/3}+p_i^{1/3}+1}\frac{1}{p_i^{1/3}-1}=\frac{q_i}{p_i-1}.$$
So $\frac{q_i}{p_i-1}=\frac{q}{p-1}$ for all $i \in \mathbb{N}$. Hence, continuing
\begin{align} 
m(y+b_i)&=p_im(y)+q_i \nonumber \\
u(e^ye^{b_i})&=p_iu(e^y) \textrm{, where } m(y)=u(e^y)+\frac{q_i}{p_i-1} \nonumber \\
u(e^y)&=\alpha e^{\beta y}\textrm{ (by Lemma~\ref{lem:srnie})}\nonumber \\
m(y)&=\alpha e^{\beta y}+\frac{q}{p-1}, \nonumber \\
g(x)&=\alpha e^{\beta f(x)} +\frac{q}{p-1}, \nonumber \\
g(x)&\simeq e^{\beta f(x)}. \nonumber
\end{align}
\subsubsection{Subcase $p=1$}
We have $q_{i+1}=\frac{q_i}{3}$, $b_{i+1}=\frac{b_i}{3}$. So $\frac{q_i}{b_i}=\frac{q}{b}$
for all $i \in \mathbb{N}$ and $m(y+b_i)=m(y)+q_i$. Moreover, similarly like above,
$p_i=1$ for an arbitrary $i$. Hence
\begin{align}
m(y+b_i)&=m(y)+q_i, \nonumber \\
m(y)&=\alpha+\beta y \textrm{ (by the remark of Lemma~\ref{lem:srnie})} \nonumber \\
g(x)&=\alpha+\beta f(x) \nonumber \\
g(x)&\simeq f(x). \nonumber
\end{align}

Concluding Case~\ref{subsec:case2} we get the family
\begin{equation}
\{ e^{\beta f(x)} \colon \beta \neq 0 \} \cup \{ f(x) \}. \label{F2}
\end{equation}

Changing $f$ in an affine way does not change the means generated by above functions.
\begin{xrem}
Like in Case~\ref{subsec:case1} the second implication, 
for an arbitrary $b$ the function $f^{-1}(f(x)+b)$
is neutral for each mean generated by the function belong to the family~\eqref{F2}, 
becomes very simple in view of~\eqref{eq:postacn}. 
\end{xrem}

\subsection{Final conclusion and remarks}
To end the proof of Theorem~\ref{thm:invscale} it is enough to observe that
Lemma~\ref{lem:fam_likePM} implies each of the family 
\eqref{F11}, \eqref{F12} and \eqref{F2} to generate a scale on $I$.

\begin{xrem}
The functions considered in Case~\ref{subsec:case1} and Case~\ref{subsec:case2} are called \emph{scaling} and \emph{translation} respectively.
\end{xrem}

\begin{xrem}
The proof may be partitioned in the different way:
Cases may depends on $p$ and subcases on $a$. 
In this situation there is technical problem to conclude each case
like \eqref{F11}, \eqref{F12} and \eqref{F2} do.
\end{xrem}

Theorem~\ref{thm:invscale} might be inverse in the following sense
\begin{proposition}
Let $I$ be an interval and $f \colon I \rightarrow \R_+$ be a continuous, strictly monotone function.
Then there exists an $f$-neutral function $\n$ such that the family 
\begin{equation}
\{ f(x)^{\beta} \colon \beta \neq 0 \} \cup \{ \ln f(x) \}
\end{equation}
generates an unique scale invariant under $\n$ and all its $f$-roots.
\end{proposition}
\begin{proof}
Let $\n(x)=f^{-1}(af(x))$ 
(for technical reason, to eliminate the problem with domain, we assume $a\approx 1$)
 and $\mathcal{F}$ denote given family of functions.

Each mean generate by a function belonging to $\mathcal{F}$
is invariant under $\n$ and all its $f$-roots.

Then, by Lemma~\ref{lem:fam_likePM}, $\mathcal{F}$ generates a scale
and, by Theorem~\ref{thm:invscale}, this scale is unique.
\end{proof}

\section{Applications}
The following proposition has been proved and reproved many times i.e. \cite[Theorem 5.1.]{WJ05}, \cite[Theorem 7.]{XL09}
\begin{proposition}
If for all $a \neq 0$ the mapping $x\mapsto x+a$ is neutral for a quasi-arithmetic mean, then it is one of the $\log-\exp$ mean (either arithmetic mean or the one generated by $e^{\beta x}$ for some $\beta \neq 0$).
\end{proposition}

\begin{proof}
It is easy to observe that the mapping is neutral for arithmetic mean, and it is considered in  Case~\ref{subsec:case2}. Hence all means invariant under $\n$ are generated by one of the functions in the family 
$\{ e^{\beta x} \colon \beta \neq 0 \} \cup \{ x \}$.
\end{proof}

Now we prove very classical result, which can be found for example in \cite[pp. 68--69]{HLP}, \cite[pp. 272--273]{bullen} and the references therein, \cite{XL09} (generalization for fuzzy sets and orness measure). 
\begin{proposition}
\label{prop:powermeans}
Power means are the only homogeneous, quasi-arithmetic means defined on $\R_{+}$.
\end{proposition}
\begin{proof}
Homogeneous means are, equivalently, invariant under mapping $\R_{+} \ni x \mapsto 2x$ and all its roots
(see Example~\ref{ex:hom}). Hence, by Theorem~\ref{thm:invscale}, the set of all homogeneous,
quasi-arithmetic means is a scale on $\R_+$.
But powers means are a scale and, moreover, all power means are homogeneous.
Hence all homogeneous, quasi-arithmetic means are just power means.
\end{proof}


\end{document}